\documentclass[reqno]{amsart}

\usepackage{amsmath}
\usepackage{amssymb}
\usepackage[margin=1in]{geometry}
\usepackage{xcolor}
\usepackage{comment}
\usepackage{graphicx}
\usepackage{empheq}

\theoremstyle{plain}
\newtheorem{thm}{Theorem}[section]
\newtheorem{lem}{Lemma}[section]
\newtheorem{prop}{Proposition}[section]

\theoremstyle{definition}
\newtheorem{defn}{Definition}[section]

\theoremstyle{remark}

\newcommand{\norm}[1]{\left\lVert#1\right\rVert}
\newcommand{\Lip}{\operatorname{Lip}}

\newcommand{\diag}[1]{\left[#1\right]}
\newcommand{\sgn}[1]{\operatorname{sgn}(#1)}

\usepackage{cite}
\usepackage{hyperref}
\hypersetup{ colorlinks = true, urlcolor = blue, linkcolor = blue, citecolor = red }
\usepackage{enumitem}
\usepackage{nameref}
\makeatletter
\let\orgdescriptionlabel\descriptionlabel
\renewcommand*{\descriptionlabel}[1]{%
  \let\orglabel\label
  \let\label\@gobble
  \phantomsection
  \edef\@currentlabel{#1\unskip}%
  \let\label\orglabel
  \orgdescriptionlabel{#1}%
}
\numberwithin{equation}{section}

\begin{document}
	
\title[ Lipschitz regularity for anisotropic equations with nonstandard growth]{Lipschitz regularity for anisotropic fully nonlinear equations with nonstandard growth}
%%%%% AUTHORS %%%%%

\author[Byun]{Sun-Sig Byun}
\address{Department of Mathematical Sciences and Research Institute of Mathematics,
	Seoul National University, Seoul 08826, Republic of Korea}
\email{byun@snu.ac.kr}

\author[Kim]{Hongsoo Kim}
\address{Department of Mathematical Sciences, Seoul National University, Seoul 08826, Republic of Korea}
\email{rlaghdtn98@snu.ac.kr}

\thanks {S.-S. Byun was supported by Mid-Career Bridging
Program through Seoul National University. H. Kim was supported by the National Research Foundation of Korea(NRF) grant funded by the Korea government [Grant No. 2022R1A2C1009312].}

\makeatletter
\@namedef{subjclassname@2020}{\textup{2020} Mathematics Subject Classification}
\makeatother
% \date{\today}
% It is required to enter 2010 MSC.
% Please provide minimum  5 keywords.

\everymath{\displaystyle}

\begin{abstract}
	We establish interior Lipschitz regularity for solutions to anisotropic fully nonlinear equations with nonstandard growth, without imposing any restriction on the gap between the highest and lowest growth exponents.
    Our proof is based on an anisotropic variant of the seminal Ishii–Lions method.
    Our result furnishes a viscosity analogue of the divergence-form theory in \cite{Bousquet20}, adapted to the non-divergence setting.
\end{abstract}

\maketitle

\section{Introduction} \label{sec1}
In this paper, we study the interior Lipschitz regularity of viscosity solutions of the following degenerate anisotropic fully nonlinear elliptic equation:
\begin{align} \label{PDE}
    F(D^2u, Du, x) = f(x,Du) \quad \text{in } B_1,
\end{align}
where $F=F(M,z,x)$ satisfies
\begin{align} \label{pigrowth}
     \lambda\norm{\diag{|z_i|^{p_i/2}}N\diag{|z_i|^{p_i/2}}} \leq F(M+N,z,x)- F(M,z,x) \leq \Lambda\norm{\diag{|z_i|^{p_i/2}}N\diag{|z_i|^{p_i/2}}}
\end{align}
for any matrix $N\geq 0$ and constants $0\leq p_1 \leq \cdots \leq p_n$, where $\diag{|z_i|^{p_i/2}}$ is the diagonal matrix with entries $|z_i|^{p_i/2}$ on the diagonal.
The main point is the present paper is that there is no assumption about the closeness of $p_1$ and $p_n$; for example, an upper bound of $p_n/p_1$ or $p_n-p_1$ which is the condition imposed in the earlier papers as in \cite{Demengel17}.

The (degenerate) anisotropic Laplacian operator, which is defined as
\begin{align} \label{piLap}
    \Delta_{(p_i)} u = \sum_i \partial_i(|\partial_iu|^{p_i-2} \partial_iu) = \sum_i (p_i-1)|\partial_iu|^{p_i-2}\partial_{ii}u
\end{align}
for $p_i \geq 0$, has been a classical research topic in the area of PDEs and extensively studied by many authors.
It is, in the variational sense, Euler-Lagrange equation of the following orthotropic energy functional
\begin{align} \label{piLap2}
    w \mapsto \int \sum_i \frac{1}{p_i} |\partial_iw|^{p_i} dx.
\end{align}
Note that the anisotropic $(p_i)$-Laplacian is more degenerate than the $p$-Laplacian, given by $$\Delta_{p} u = \sum_i \partial_i(|\nabla u|^{p-2} \partial_iu).$$
While the $p$-laplacian degenerates only when $|\nabla u|=0$, the anisotropic  $(p_i)$-Laplacian degenerates when $\partial_iu=0$ for some $i$. 
This makes the analysis more difficult and delicate since the equation can be degenerate where the gradient is large enough.
Thus, even for the pseudo $p$-Laplacian case when $p_i$ are all the same, the Lipschitz regularity was not known until recently;
Bousquet, Brasco, Leone, and Verde \cite{Bousquet18} proved the Lipschitz regularity for the weak solution. See also \cite{Bousquet16,Lindquist18,Brasco17}.
Furthermore, Demengel \cite{Demengel162} proved the Lipschitz regularity for the viscosity solution, and also considered a more general class of the Pucci operators together with Birindelli \cite{Demengel16} using the Ishii-Lions method.

Even if the anisotropic degeneracy is ignored, the equation still exhibits a nonstandard growth condition due to the variability of $(p_i)$.
The nonstandard growth condition means that 
\begin{align} \label{nonstand}
    w \mapsto \int F(Dw) dx, \quad \text{ where } \quad |z|^{p_1} \lesssim F(z) \lesssim |z|^{p_1}+|z|^{p_n},
\end{align}
which was considered by Marcellini \cite{Marcellini89,Marcellini91}.
If $p_n / p_1$ is large, then there exist well-known unbounded counterexamples by Giaquinta \cite{Giaquinta87}, Marcellini \cite{Marcellini89,Marcellini91} and Hong \cite{Hong92}. 
For the boundedness for weak solutions of \eqref{piLap}, we require that the gap between $p_1$ and $p_n$ be small, and the known optimal condition, established by Fusco and Sbordone \cite{Fusco90}, is $p_n \leq \overline{p}^\star$ where 
\begin{align*}
    \frac{1}{\overline{p}} = \frac{1}{n}\sum_i\frac{1}{p_i}, \quad \overline{p}^\star = \frac{n\overline{p}}{n-\overline{p}}.
\end{align*}
Cupini, Marcellini and Mascolo \cite{Cupini15,Cupini17} extended this result to more general functionals. See also \cite{Dibenedetto16, Fusco93, Cianchi00,Granucci19, Cupini09,Bildhauer07,Bhattacharya96,Canale01,Baroni17,Liskevich09}.
On the other hand, if we only consider bounded solutions, it is interesting to see what is the condition on $p_1$ and $p_n$ in order to get the regularity result.

Surprisingly, assuming the boundedness of the solution a priori, Bousquet and Brasco \cite{Bousquet20} recently proved that 
the weak solution of \eqref{piLap} is Lipschitz continuous without any assumption on the closeness of $p_1$ and $p_n$.
The method used in \cite{Bousquet20} is based on a Caccioppoli-type inequality, Moser-type iteration and higher integrability of each directional derivative with multiple iterative scheme, which cannot be applicable to viscosity solutions under consideration in this paper.
See also \cite{Bousquet24} for the singular case.

On the other hand, from the viewpoint of viscosity solutions, Demengel \cite{Demengel17} proved the Lipschitz regularity of viscosity solutions of \eqref{PDE} under the anisotropic $(p_i)$ assumption \eqref{pigrowth} and some structural conditions by using the Ishii-Lions method. However, the condition
\begin{align} \label{p1pn}
    p_n-p_1 < 1
\end{align}
appears to be essential in their approach.
Since a viscosity solution is defined only for continuous functions and is therefore automatically bounded, we believe that the condition \eqref{p1pn} can be removed, as in the case of bounded weak solutions.
We believe that condition \eqref{p1pn} arises from the nonstandard growth condition \eqref{nonstand} for bounded solutions, as discussed in \cite{Choe92,Cupini14,Kim25}.
The main difficulty in dealing with nonstandard growth lies in the ellipticity ratio, defined by
\begin{align} \label{elpratio}
    \mathcal{R}_F(z) := \frac{\text{highest eigenvalue of } \partial_{zz}F(z)}{\text{lowest eigenvalue of } \partial_{zz}F(z)}\lesssim 1+|z|^{p_n-p_1},
\end{align}
which can blow up when $|z| \rightarrow \infty$ potentially leading to irregularity of the solution when $p_n-p_1$ is large.
However, if we consider the ellipticity ratio of \eqref{piLap2} in each direction, then
\begin{align} \label{elpratio2}
    \mathcal{R}_F^{i}(z) := \frac{ \max \{ \partial_{z_iz_i}F(z) \} }{ \min \{ \partial_{z_iz_i}F(z) \} }\lesssim C
\end{align}
which remains bounded regardless of the values of $p_1$ and $p_n$ is, even though the growth of each directional derivative $\partial_{z_iz_i}F(z)$ may differ. Therefore, by carefully treating each direction of the anisotropic equation separately, without mixing them, one may obtain regularity results without relying on the condition \eqref{p1pn}.
In fact, in \cite{Demengel17} every function used in the proof via the Ishii-Lion method is isotropic.
As a result, without incorporating anisotropic information, the issue of the nonstandard growth ellipticity ratio \eqref{elpratio} becomes unavoidable, and the assumption \eqref{p1pn} remains necessary.
Therefore, in this work, we consider anisotropic test functions within the Ishii–Lions framework to take advantage of the directional ellipticity bound \eqref{elpratio2} and establish Lipschitz regularity without relying on the assumption \eqref{p1pn}. 

We state the main assumptions on \eqref{PDE} as follows.
For any $x,y \in B_1$, $z\in \mathbb{R}^n$ and $M,N \in S(n)$, we assume
\begin{description}
    \item[(A1)\label{A1}] $F$ is anisotropic $(p_i)$-growth with $0\leq p_1 \leq \cdots \leq p_n$ in the sense that there holds
    \begin{align*}
        \mathcal{M}^-_{\lambda,\Lambda}(\diag{|z_i|^{p_i/2}}N\diag{|z_i|^{p_i/2}}) \leq F(M+N,z,x) - F(M,z,x) \leq \mathcal{M}^+_{\lambda,\Lambda}(\diag{|z_i|^{p_i/2}}N\diag{|z_i|^{p_i/2}})
    \end{align*}
    with some constants $0<\lambda< \Lambda$.
    \item[(A2)\label{A2}] For any $h \in \mathbb{R}$, $F$ satisfies
    \begin{align*}
    |F(M,z+he_i,x) - F(M,z,x)| \leq& \Lambda\left||z_i+h|^{p_i}-|z_i|^{p_i}\right||M_{ii}| \\
    +&\Lambda \sum_{j\neq i}\left(\left||z_i+h|^{p_i/2}-|z_i|^{p_i/2}\right||z_j|^{p_j/2}|M_{ij}|\right).
    \end{align*} 
    \item[(A3)\label{A3}] $F$ satisfies
    \begin{align*}
    |F(M,z,x) - F(M,z,y)| \leq \Lambda|x-y|^\chi\sum_{i,j}|z_i|^{p_i/2}|z_j|^{p_j/2}|M_{ij}|,
    \end{align*}
    for some $0<\chi\leq 1$.
    \item[(A4)\label{A4}] $f \in C(\Omega \times \mathbb{R}^n)$ with
    \begin{align*}
        |f(x,z)| \leq C_f(1+\sum_i|z_i|^{p_i+1}),
    \end{align*}
    for some constant $C_f > 0$.
\end{description}

Some examples of models satisfying the above assumptions are:
\begin{enumerate}
    \item the anisotropic $(p_i)$-Laplacian equation \eqref{piLap} with H\"older continuous coefficients:
    \begin{align} \label{dp}
        \sum_i a_i(x)|\partial_iu|^{p_i}\partial_{ii}u = f(x),
    \end{align}
    where $a_i \in C^\chi$ for some $\chi>0$ and $1 \leq a_i(x) \leq \Lambda$.
    \item the anisotropic Pucci operator in \cite{Demengel16}:
    \begin{align} \label{pq}
      \mathcal{M}^\pm_{\lambda,\Lambda} \left( \diag{|D_iu|^{p_i/2}} D^2u \diag{|D_iu|^{p_i/2}} \right) =f(x),
    \end{align}
    for some $0<\lambda\leq \Lambda$.
    \item More generally,
    \begin{align} \label{Gwithpi}
        F(D^2u,Du,x) := G\left( \diag{a_i(x)|D_iu|^{p_i/2}}D^2u\diag{a_i(x)|D_iu|^{p_i/2}}\right) =f(x),
    \end{align}
     where $a_i \in C^\chi$ with $\chi>0$, $1 \leq a_i(x) \leq \Lambda$ and $G=G(M)$ is uniformly elliptic in the sense that $$ \mathcal{M}^-_{\lambda,\Lambda}(N)\leq G(M+N)-G(M) \leq \mathcal{M}^+_{\lambda,\Lambda}(N) $$ for any $M,N \in S(n)$.
\end{enumerate}

For $0\leq p_1 \leq \cdots \leq p_m \leq\cdots \leq p_n$, we denote by $p_m$ to be the smallest number which is not $0$.
We state the main theorem below.
\begin{thm} \label{Main}
    Let $u \in C(B_1)$ be a viscosity solution of the anisotropic problem
    \begin{align*}
        F(D^2u,Du,x) = f(x,Du) \quad \text{in } \ B_1
    \end{align*}
    under the assumptions \ref{A1}-\ref{A4}.
    Then $u \in \Lip(B_{1/2})$ with the estimate
    \begin{align*}
        \norm{u}_{\Lip(B_{1/2})} \leq C(n,\lambda, \Lambda, p_m, p_n, C_f, \norm{u}_{L^\infty(B_1)}).
    \end{align*}
\end{thm}
We gives some remarks of the theorem.
\begin{itemize}
\item The assumption \ref{A1} was first considered by Demengel in \cite{Demengel17}.
Observe that \eqref{pigrowth} is equivalent to the assumption \ref{A1}.
However, the assumption \ref{A2} and \ref{A3} differ from those in \cite{Demengel17}.
In that work,
the following conditions 
\begin{align*}
    |F(M,z+he_i,x) - F(M,z,x)| \leq \Lambda ||z_i+h|^{p_i}-|z_i|^{p_i}|M|,
\end{align*}
and
\begin{align*}
    |F(M,z,x) - F(M,z,y)| \leq \Lambda|x-y|^\chi\sum_{i}|z_i|^{p_i}|M|.
\end{align*}
were considered instead of \ref{A2} and \ref{A3}, respectively.
Observe that the two assumptions above involve $|M|$,  which can be interpreted as information about the combined behavior across all directions. However, in the anisotropic setting where the growth may differ by direction, it is necessary to extract directional information. Therefore, we instead consider assumptions involving the individual components $|M_{ij}|$. 
\item In fact, for weak solutions \cite{Bousquet20}, the authors consider only the case with homogeneous right-hand side, $f=0$.
It appears that, by employing the similar approach used in \cite{Bousquet20}, one could obtain Lipschitz regularity with $f\neq0$, though at the cost of substantial additional computation.
On the other hand, our approach here allows for handling the nonhomogeneous case $f\neq0$ more easily, but only when $ f\in L^\infty$, not $ L^\gamma$ for some $\gamma>n$.
\end{itemize}

We now introduce main ideas and novelties of our proof.
We intend to design a suitable anisotropic test function so that each direction exhibits even growth and compute each directional component separately without mixing them.
To this end, we employ the celebrated Ishii–Lions technique developed in \cite{Ishii92}.
First, we prove that $u$ is $C^\gamma$ for any $\gamma<1$ close to 1, and then prove the Lipschitz continuity in light of $C^\gamma$-continuity.
Recall that in the Ishii-Lions technique, the main goal is to show that, for sufficiently large constants $L,K>0$ 
\begin{align*}
        M= \max_{x,y\in \overline{B_1}} \left\{ u(x) - u(y) - L\phi(|x-y|) - \frac{K}{2}|x-x_0|^2 -\frac{K}{2}|y-x_0|^2 \right\} \leq 0,
    \end{align*}
where $\phi(t)=t^\gamma$, in order to obtain $\gamma$-H\"older regularity.
This can be proved by contradiction.
Suppose $M>0$, and let $x,y$ be the points at which the maximum is attained and set $a=x-y$.
We first consider $\phi$ as a test function and substitute $D\phi(|a|),D^2\phi(|a|)$ into the equation in place of $Du, D^2u$, respectively.
In doing so, we encounter terms of the form $|D_i\phi(|a|)|^{p_i}|D_{ii}\phi(|a|)|$, but due to the variation in the exponents $p_i$, these quantities are not balanced across directions. 
To address this issue, we instead employ an anisotropic distance function with direction-dependent growth rates $\alpha=(\alpha_i)$, defined by
\begin{align*}
    |a|_\alpha^2 = \frac{1}{\alpha_1}|a_1|^{2\alpha_1} + \cdots+ \frac{1}{\alpha_n}|a_n|^{2\alpha_n}, \quad \alpha_i = \frac{1+p_i/2}{1+\gamma p_i/2}
\end{align*}
instead of the standard isotropic distance function $|a|^2=|a_1|^2 + \cdots |a_n|^2$.
Then, by direct computation, the terms $|D_i\phi(|a|_\alpha)|^{p_i}|D_{ii}\phi(|a|_\alpha)|$ become balanced across directions, and therefore the closeness condition between $p_1$ and $p_n$ is no longer required.

Moreover, we need to find the the Hessians $X=D^2u(x)$ and $Y=D^2u(y)$, which can be achieved by applying the Ishii–Lions lemma to transfer the estimates of $Z=D^2\phi(|x|_\alpha)$ to $X,Y$.
Explicitly, for any $\epsilon>0$, we have
\begin{align*}
    -\left(\frac{1}{\epsilon} + \norm{Z}\right) 
    \begin{pmatrix}
        I & 0 \\
        0 &I
    \end{pmatrix}
    \leq
    \begin{pmatrix}
        X & 0 \\
        0 &-Y
    \end{pmatrix}
    \leq
    \begin{pmatrix}
        Z^\epsilon & -Z^\epsilon \\
        -Z^\epsilon & Z^\epsilon
    \end{pmatrix},
\end{align*}
where $Z^\epsilon=Z+2\epsilon Z^2$.
See \cite{Ishii92, Katzourakis15} for the detailed explanation of the Ishii-Lion lemma.
By choosing $\epsilon \lesssim 1/\norm{Z}$, we obtain $\norm{X}, \norm{Y} \lesssim \norm{Z}$.
However, in the anisotropic setting, we require componentwise estimates for $X_{ij}$ and $Y_{ij}$ due to the assumption \ref{A2} and \ref{A3}.
In particular, we need estimates of the form $|X_{ii}| \lesssim |Z_{ii}|$ so that the quantity $|D_i\phi(|a|_\alpha)|^{p_i}|X_{ii}|$ remains balanced across directions.
By applying $(e_i,0)$ to the above matrix inequalities, we obtain only $|X_{ii}| \lesssim \norm{Z} $.
However, since $\norm{Z}$ combines effects from all directions while  $Z$ grows anisotropically, it fails to capture the precise behavior of the individual component $|Z_{ii}|$.
As a result, this approach fails to yield the desired estimate $|X_{ii}| \lesssim |Z_{ii}|$.

In order to overcome this obstacle, we instead consider the anisotropic Ishii-Lions Lemma (Lemma \ref{IL}).
Rather than using a scalar parameter $\epsilon>0$, we introduce a direction-dependent vector $\epsilon=(\epsilon_i)$ where each $\epsilon_i>0$ may vary with each direction.
Then we have
\begin{align*}
    -(1+\norm{\diag{\sqrt{\epsilon_i}}Z\diag{\sqrt{\epsilon_i}}}) \begin{pmatrix}
        \diag{1/\epsilon_i} & 0 \\
        0 & \diag{1/\epsilon_i}
    \end{pmatrix}
    \leq 
    \begin{pmatrix}
        X & 0 \\
        0 &-Y
    \end{pmatrix}
    \leq 
    \begin{pmatrix}
        Z^\epsilon & -Z^\epsilon \\
        -Z^\epsilon & Z^\epsilon
    \end{pmatrix},
\end{align*}
where $Z^\epsilon = Z + 2Z\diag{\epsilon_i}Z$.
Then by choosing $\epsilon_i \lesssim 1/|Z_{ii}|$ and applying $(e_i,0)$ or $(0,e_i)$ to the above inequalities, we get the desired componentwise  estimates $|X_{ii}|, |Y_{ii}| \lesssim |Z_{ii}|$.

The paper is organized as follows.
In Section \ref{sec2} we specify the notations, introduce some preliminary lemmas.
In Section \ref{sec3} we prove the $\gamma$-H\"older regularity of the solution for any $\gamma<1$ close to 1.
In Section \ref{sec4} we prove the Lipschitz regularity of the solution and finish the proof of Theorem \ref{Main}.
In Section \ref{sec5} we prove the anisotropic Ishii-Lions lemma (Lemma \ref{IL}) by a simple anisotropic scaling of the original Ishii-Lions lemma.

\section{Notations and Preliminaries} \label{sec2}

Throughout the paper, we write $B_r(x_0) = \{ x \in \mathbb{R}^n : |x-x_0| <r \}$ and $B_r = B_r(0)$.
$S(n)$ denotes the space of symmetric $n \times n$ real matrices and $I$ denotes the identity matrix.

For some $a=(a_i) \in \mathbb{R}^n$, we always write the diagonal matrix with entries $a_i$ as following.
\begin{align*}
    \diag{a_i} := \operatorname{diag}(a_1,\cdots,a_n) \in S(n).
\end{align*}
We also write
\begin{align*}
    \diag{a_i,a_i} := \operatorname{diag}(a_1,\cdots,a_n,a_1,\cdots,a_n) \in S(2n).
\end{align*}
As usual, the sign function $\operatorname{sgn} : \mathbb{R}\rightarrow \{ -1,0,1\} $ is defined as $\sgn{t}:=\frac{t}{|t|}$ for $t \neq 0$ and $\sgn{0}=0$.

We state the definition of subjet and superjet introduced in \cite{Ishii92}, which are used in the definition of viscosity solution below.
\begin{defn}
    For any continuous function $u \in C(\Omega)$ and $x \in \Omega$, we define the superjet and subjet by
    \begin{align*}
        \mathcal{J}^{2,+}_\Omega u(x) = \left\{ (p,X) \in \mathbb{R}^n \times S(n) : u(x+h) \leq u(x) + \langle p, h \rangle  + \frac{1}{2} \langle Xh, h\rangle +o(h^2), \ \ \forall h \in \mathbb{R}^n \right\}, \\
        \mathcal{J}^{2,-}_\Omega u(x) = \left\{ (p,X) \in \mathbb{R}^n \times S(n) : u(x+h) \geq u(x) + \langle p, h \rangle  + \frac{1}{2} \langle Xh, h\rangle +o(h^2), \ \ \forall h \in \mathbb{R}^n \right\}.
    \end{align*}
    Furthermore, we define the closed superjet and subjet by
    \begin{align*}
        \overline{\mathcal{J}^{2,\pm}_\Omega} u(x) = \left\{ (p,X)  : \exists x_n \in \Omega, \exists (p_n,X_n) \in \mathcal{J}^{2,\pm}_\Omega u(x_n), \ (x_n, u(x_n),p_n,X_n) \rightarrow (x,u(x),p,X)\right\}.
    \end{align*}
\end{defn}

Now we recall the definition of the inequalities \eqref{PDE} in the viscosity sense from \cite{Ishii92, Caffarelli95} as follows.
\begin{defn}
Let $f \in C(B_1)$.
We say that $u \in C(\overline{B}_1)$ satisfies 
$$F(D^2u,Du,x) \leq f(x,Du) \quad\text{in}\ B_{1} \quad (\text{resp.} \geq) $$
in the viscosity sense, if for any $x_0 \in B_1$ and test function $\psi \in C^2(B_1)$ such that $u-\psi$ has a local minimum (resp. maximum) at $x_0$, then
$$F(D^2\psi(x_0),D\psi(x_0),x_0)\leq f(x_0,D\psi(x_0)) \quad (\text{resp.} \geq).$$ 
\end{defn}

We also recall the definition and some properties of the Pucci operator (see \cite{Caffarelli95}).
\begin{defn}
    For given $0<\lambda \leq \Lambda$, we define the Pucci operators $\mathcal{P}_{\lambda,\Lambda}^{\pm} : S(n) \rightarrow \mathbb{R}$ as follows:
\begin{align*}
    \mathcal{M}_{\lambda,\Lambda}^{+}(M) := \lambda \sum_{e_i(M)<0}e_i(M) + \Lambda \sum_{e_i(M)>0}e_i(M), \\
    \mathcal{M}_{\lambda,\Lambda}^{-}(M) := \Lambda \sum_{e_i(M)<0}e_i(M) + \lambda \sum_{e_i(M)>0}e_i(M),
\end{align*}
where $e_i(M)$'s are the eigenvalues of $M$.
\end{defn}
For convenience we abbreviate as $\mathcal{M}^{\pm}:=\mathcal{M}_{\lambda,\Lambda}^{\pm}$.
We also write some properties of the Pucci operator (See \cite{Caffarelli95}).
\begin{prop} 
For any $M,N \in S(n)$, we have
    \begin{enumerate}
        \item $ \mathcal{M}_{\lambda,\Lambda}^{-}(M) +\mathcal{M}_{\lambda,\Lambda}^{-}(N)\leq \mathcal{M}_{\lambda,\Lambda}^{-}(M+N) \leq \mathcal{M}_{\lambda,\Lambda}^{-}(M) +\mathcal{M}_{\lambda,\Lambda}^{+}(N)$. \\
        \item $ \mathcal{M}_{\lambda,\Lambda}^{+}(M) +\mathcal{M}_{\lambda,\Lambda}^{-}(N)\leq \mathcal{M}_{\lambda,\Lambda}^{+}(M+N) \leq \mathcal{M}_{\lambda,\Lambda}^{+}(M) +\mathcal{M}_{\lambda,\Lambda}^{+}(N).$ 
    \end{enumerate}
\end{prop}
Moreover, we introduce the anisotropic Ishii-Lions Lemma.
\begin{lem} \label{IL} (The anisotropic Ishii-Lions Lemma)
    Let $u,v \in C(\Omega)$ and $\phi(x,y) \in C^2(\Omega\times \Omega)$.
    Assume that $(\overline{x},\overline{y}) \in \Omega\times\Omega$ is a local maximum point of $u(x)+v(y) -\phi(x,y)$.
    Then for any $\epsilon=(\epsilon_i)$ with $\epsilon_i>0$ , there exist $X,Y \in S(n)$ such that
    \begin{align*}
        (D_x\phi(\overline{x},\overline{y}),X) \in \overline{\mathcal{J}^{2,+}_\Omega} u(\overline{x}), \quad (D_y\phi(\overline{x},\overline{y}),Y) \in \overline{\mathcal{J}^{2,-}_\Omega} v(\overline{y})
    \end{align*}
    and
    \begin{align*}
        -(1+\norm{\diag{\sqrt{\epsilon_i},\sqrt{\epsilon_i}}Z\diag{\sqrt{\epsilon_i},\sqrt{\epsilon_i}}}) \begin{pmatrix}
             \diag{1/\epsilon_i} & 0 \\
             0 & \diag{1/\epsilon_i}
         \end{pmatrix}
         \leq 
        \begin{pmatrix}
             X & 0 \\
             0 & Y
         \end{pmatrix}
         \leq 
         Z + Z\diag{\epsilon_i,\epsilon_i}Z,
    \end{align*}
    where $Z = D^2\phi(\overline{x},\overline{y})$.
\end{lem}
Observe that if each $\epsilon_i$ is the same, then this result coincides with the original Ishii-Lions Lemma (see \cite{Ishii92,Silvestre13, Katzourakis15} and Lemma \ref{ILo}).
The proof of the above lemma can be found later in the Section \ref{sec5}.

Finally, we introduce some functions to deal with anisotropic structure.
For $z\in \mathbb{R}^n$ and $\alpha=(\alpha_i)$ with $\alpha_i\geq 1$, we define an anisotropic distance function as
\begin{align*}
    |z|_\alpha := \sqrt{\frac{1}{\alpha_1}|z_1|^{2\alpha_1} + \cdots+ \frac{1}{\alpha_n}|z_n|^{2\alpha_n}}. 
\end{align*}
Note that this function is $C^{1,1}(\mathbb{R}^n\setminus\{0\})$.
By simple calculation, we have
\begin{gather*} 
    D_i|z|_\alpha = \sgn{z_i} \frac{ |z_i|^{2\alpha_i-1}}{|z|_\alpha} , \\
    D^2|z|_\alpha = \frac{1}{|z|_\alpha}\diag{\sgn{z_i}|z_i|^{\alpha_i-1}}\left( \diag{2\alpha_i-1}- \left(\frac{|z_i|^{\alpha_i}}{|z|_\alpha}\right)  \otimes  \left(\frac{|z_i|^{\alpha_i}}{|z|_\alpha} \right)\right)\diag{\sgn{z_i}|z_i|^{\alpha_i-1}}.
\end{gather*}

\section{H\"older estimates for anisotropic equations} \label{sec3}
In this section, we first prove $\gamma$-H\"older continuity of a solution under consideration for any $\gamma<1$ close to $1$ by using the anisotropic Ishii-Lions method.
\begin{proof}[Proof of H\"older estimates.] 
    We fix $x_0 \in B_{1/2}$ and claim that
    \begin{align*}
        M= \max_{x,y\in \overline{B_1}} \left\{ u(x) - u(y) - \phi(|L^{\beta} (x-y)|_\alpha) - \frac{K}{2}|x-x_0|^2 -\frac{K}{2}|y-x_0|^2 \right\} \leq 0,
    \end{align*}
    where $\phi(t) = t^\gamma/\gamma$ for some $1/2<\gamma<1$ close to 1, $\alpha=(\alpha_i)$ and $\beta=(\beta_i)$ defined as
    \begin{align} \label{ab}
        \alpha_i = \frac{1+p_i/2}{1+\gamma p_i/2}, \qquad \beta_i = \frac{1/2}{1+p_i/2},
    \end{align}
    for positive constant $L,K>0$ large enough.
    Here we use an abuse of notation as
    \begin{align*}
        L^{\beta} (x-y) := (L^{\beta_1}(x_1-y_1), \cdots,L^{\beta_n}(x_n-y_n)).
    \end{align*}
    Note that $1\leq \alpha_i <1/\gamma$.  Since $\phi(|\cdot|_\alpha)$ is $C^\gamma$, $u$ is $C^\gamma$ provided the claim is true.
    We argue by contradiction by assuming $M>0$ for any large $L, K>0$.
    From now on, we write $x,y \in \overline{B_1}$ as points where the maximum $M$ is attained.
    For convenience, we write $a=(a_i)$ and $\delta=(\delta_i)$ as
    \begin{align*}
        a=L^{\beta} (x-y), \quad \delta=x-y.
    \end{align*}
    Then $x\neq y$ and 
    \begin{align*}
        \phi(|a|_\alpha)+\frac{K}{2}|x-x_0|^2 +\frac{K}{2}|y-x_0|^2 \leq 2\norm{u}_{L^\infty}.
    \end{align*}
    We fix $K = 16\norm{u}_{L^\infty}$ so that $|x-x_0| < 1/2$ and $ |y-x_0| < 1/2$, which implies $x,y \in B_1$.
    Also, we get 
    \begin{align} \label{adelta}
       0 <|a|_\alpha \leq (2\norm{u}_{L^\infty})^{1/\gamma}, \quad 0\leq \delta_i \leq \frac{C\norm{u}_{L^\infty}^{1/\alpha_i\gamma}}{L^{\beta_i}}.
    \end{align}
    Note that $a \neq0$ since $x\neq y$, but $a_i$ can be $0$ for some $i$.
     Now we apply the anisotropic Ishii-Lions lemma (Lemma \ref{IL}) to $u(x)- \frac{L_2}{2}|x-x_0|^2$ and $-u(y) - \frac{L_2}{2}|y-x_0|^2$. Then there exist $X,Y \in S(n)$ such that
     \begin{align*}
          (z, X)\in \overline{\mathcal{J}}^{2,+} \left( u(x)- \frac{L_2}{2}|x-x_0|^2\right), \qquad 
          (z, Y)\in \overline{\mathcal{J}}^{2,-} \left( u(y)+ \frac{L_2}{2}|y-x_0|^2\right),
     \end{align*}
     where 
     \begin{align} \label{zest}
         z_i= D_i(\phi(|L^{\beta} (x-y)|_\alpha)) = \sgn{a_i} L^{\beta_i} \frac{|a_i|^{2\alpha_i-1}}{|a|_\alpha^{2-\gamma}} .
     \end{align}
     Thus, we have
     \begin{align*}
         (z_x, X+KI)\in \overline{\mathcal{J}}^{2,+} \left( u(x)\right), \qquad
         (z_y, Y-KI)\in \overline{\mathcal{J}}^{2,-} \left( u(y)\right),
     \end{align*}
     where
     \begin{align*}
         z_x = z+K(x-x_0)\quad \text{and} \quad z_y = z-K(y-x_0).
     \end{align*}
     Moreover, for any $\epsilon_i>0$ , we can choose $X,Y \in S(n)$ satisfying 
    \begin{align} \label{mtx}
        -(1+\norm{\diag{\sqrt{\epsilon_i}}Z\diag{\sqrt{\epsilon_i}}}) \begin{pmatrix}
             \diag{1/\epsilon_i} & 0 \\
             0 & \diag{1/\epsilon_i}
         \end{pmatrix}
         \leq 
        \begin{pmatrix}
             X & 0 \\
             0 &-Y
         \end{pmatrix}
         \leq 
         \begin{pmatrix}
             Z^\epsilon & -Z^\epsilon \\
             -Z^\epsilon & Z^\epsilon
         \end{pmatrix},
    \end{align}
     where
     \begin{align} \label{Zmtx}
         Z &= D^2(\phi(|L^{\beta} (x-y)|_\alpha)) \nonumber \\
         &=\frac{1}{|a|_\alpha^{2-\gamma}}\diag{\sgn{a_i}L^{\beta_i} |a_i|^{\alpha_i-1}} \left( \diag{(2\alpha_i-1)}-(2-\gamma) \left(\frac{|a_i|^{\alpha_i}}{|a|_\alpha}\right) \otimes \left( \frac{|a_i|^{\alpha_i}}{|a|_\alpha} \right)\right)\diag{\sgn{a_i}L^{\beta_i} |a_i|^{\alpha_i-1}},
     \end{align}
     and $Z^\epsilon = Z + 2Z\diag{\epsilon_i}Z$.
     Observe that 
     \begin{align*}
         A=\left(\diag{(2\alpha_i-1)}-(2-\gamma)\left( \frac{|a_i|^{\alpha_i}}{|a|_\alpha} \right) \otimes \left( \frac{|a_i|^{\alpha_i}}{|a|_\alpha} \right)\right)
     \end{align*}
     is a bounded matrix in the sense that $|A_{ij}|\leq C$ for all $i,j$.
     
     By the definition of viscosity solution, we have
     \begin{align*}
         F(X+KI, z_x, x) \geq f(x,z_x) \quad \text{and} \quad F(Y-KI, z_y, y) \leq f(y,z_y),
     \end{align*}
     and so
     \begin{align} \label{allest}
     \begin{split}
         -I_4:=-C_f(2+\sum|(z_x)_i|^{p_i+1} + |(z_y)_i|^{p_i+1}) &\leq f(x,z_x)-f(y,z_y) \\
         &\leq F(X+KI, z_x, x) -F(Y-KI, z_y, y) \\
         &= F(X+KI,z_x,x)-F(X+KI,z,x) \ (:=I_{21})\\ &+ F(X+KI,z,x)-F(X+KI,z,y) \ (:=I_3)\\ 
         &+ F(X+KI,z,y)-F(Y-KI,z,y) \ (:=I_1) \\
         &+ F(Y-KI,z,y)-F(Y-KI,z_y,y)  \ (:=I_{22}) \\
         &= I_{21} + I_3 + I_1 + I_{22}.
     \end{split}
     \end{align}
We will estimate $I_1, I_{21}, I_{22}, I_3, I_4$ to reach a contradiction.

\textbf{Estimates of $I_1$.}
We first find the estimate for $I_1$. By the assumption \ref{A1},
\begin{align*}
    I_1 &= F(X+KI,z,y)-F(Y-KI,z,y) \\
    &\leq \mathcal{M}^+(\diag{|z_i|^{p_i/2}}(X-Y+2KI)\diag{|z_i|^{p_i/2}}) \\
    &\leq \mathcal{M}^+(\diag{|z_i|^{p_i/2}}(X-Y)\diag{|z_i|^{p_i/2}}) + \sum_i 2\Lambda K|z_i|^{p_i}.
\end{align*}
Note that by applying $(b,b)$ with any vector $b \in \mathbb{R}^n$ for the matrix inequality \eqref{mtx}, we have $X-Y\leq 0$.
Thus, $\diag{|z_i|^{p_i/2}}(X-Y)\diag{|z_i|^{p_i/2}} \leq 0$ and all the eigenvalues are not positive.
Observe that
\begin{align*}
    \diag{\sgn{a_i}L^{\beta_i} (a_i)^{\alpha_i-1}}\diag{|z_i|^{p_i/2}} &=\diag{\sgn{a_i}L^{\beta_i} (a_i)^{\alpha_i-1} L^{\beta_i p_i/2}\frac{|a_i|^{(2\alpha_i-1)p_i/2}}{|a|_\alpha^{(2-\gamma)p_i/2}}} \\
    &=\diag{\sgn{a_i} L^{\beta_i(1+p_i/2)} \frac{|a_i|^{(2\alpha_i-1)(p_i/2) + \alpha_i-1}}{|a|_\alpha^{(2-\gamma)p_i/2}}} \\
    &=\diag{\sgn{a_i} L^{1/2} \frac{|a_i|^{\alpha_i k_i}}{|a|_\alpha^{k_i}}}, 
\end{align*}
where $k_i = (2-\gamma)p_i/2$.
Note that by definition of $\alpha=(\alpha_i)$ in \eqref{ab}, we have $$(2\alpha_i-1)(p_i/2) + \alpha_i -1 =  \alpha_i (2-\gamma)p_i/2 = \alpha_ik_i. $$
Therefore, we obtain
\begin{align*}
    \diag{|z_i|^{p_i/2}}Z\diag{|z_i|^{p_i/2}} = \frac{L}{|a|_\alpha^{2-\gamma}}\diag{\sgn{a_i} \frac{|a_i|^{\alpha_i k_i}}{|a|_\alpha^{k_i}}} A \diag{\sgn{a_i} \frac{|a_i|^{\alpha_i k_i}}{|a|_\alpha^{k_i}}}.
\end{align*}
Similarly to \cite{Demengel17}, for some $C_0>n$ to be determined, we define
\begin{align*}
    I = I(a,C_0) = \left\{ i \in \{ 1,\cdots,n\} : \frac{1}{\alpha_i}|a_i|^{2\alpha_i} \geq \frac{1}{C_0}|a|_\alpha^2\right\}.
\end{align*}
Since $|a|^2_\alpha = \sum_i \frac{1}{\alpha_i}|a_i|^{2 \alpha_i}$, $I$ is nonempty for $C_0>n$.
We also define
\begin{align*}
    w= \sum_{i \in I} \frac{\sgn{a_i}}{\alpha_i} \frac{|a|_\alpha^{k_i-1}}{|a_i|^{\alpha_i(k_i-1)}} e_i.
\end{align*}
Then since
$\diag{\sgn{a_i} \frac{|a_i|^{\alpha_i k_i}}{|a|_\alpha^{k_i}}}w = \sum_{i\in I} \frac{1}{\alpha_i} \frac{|a_i|^{\alpha_i}}{|a|_\alpha}e_i$, we have

\begin{align*}
    w^t\diag{|z_i|^{p_i/2}}Z\diag{|z_i|^{p_i/2}}w &= \frac{L}{|a|_\alpha^{2-\gamma}} \left(\sum_{i\in I} \frac{1}{\alpha_i} \frac{|a_i|^{\alpha_i}}{|a|_\alpha}e_i\right)^t A \left(\sum_{i\in I} \frac{1}{\alpha_i} \frac{|a_i|^{\alpha_i}}{|a|_\alpha}e_i\right) \\
    &=\frac{L}{|a|_\alpha^{2-\gamma}}\left( \sum_{i\in I} \frac{2\alpha_i-1}{\alpha_i^2}\frac{|a_i|^{2\alpha_i}}{|a|_\alpha^2}  - (2-\gamma) \left( \sum_{i\in I} \frac{1}{\alpha_i} \frac{|a_i|^{2\alpha_i}}{|a|_\alpha^2}\right)^2\right) \\
    &= \frac{L}{|a|_\alpha^{2-\gamma}} \frac{1}{|a|^4_\alpha}\left( \sum_{i\in I} \frac{2\alpha_i-1}{\alpha_i^2}|a_i|^{2\alpha_i} \sum_{j \in I} \frac{1}{\alpha_j}|a_j|^{2\alpha_j} - (2-\gamma)\sum_{i\in I} \frac{1}{\alpha_i}|a_i|^{2\alpha_i}\sum_{j\in I} \frac{1}{\alpha_j}|a_j|^{2\alpha_j} \right)  \\
    &+ \frac{L}{|a|_\alpha^{2-\gamma}} \frac{1}{|a|^4_\alpha} \sum_{i\in I} \frac{2\alpha_i-1}{\alpha_i^2}|a_i|^{2\alpha_i} \sum_{j \notin I} \frac{1}{\alpha_j}|a_j|^{2\alpha_j} \ (=:B)  \\
    &= \frac{L}{|a|_\alpha^{2-\gamma}} \frac{1}{|a|^4_\alpha} \left( \sum_{i,j \in I}   \left(\frac{2\alpha_i-1}{\alpha_i^2\alpha_j} - \frac{2-\gamma}{\alpha_i\alpha_j}\right) |\alpha_i|^{2\alpha_i}|\alpha_j|^{2\alpha_j}\right) +B \\
    &= \frac{L}{|a|_\alpha^{2-\gamma}} \frac{1}{|a|^4_\alpha} \left( \sum_{i,j \in I}   \frac{\gamma-1/\alpha_i}{\alpha_i \alpha_j} |\alpha_i|^{2\alpha_i}|\alpha_j|^{2\alpha_j}\right) +B.
\end{align*}
Since $1/\alpha_i -\gamma =\frac{1-\gamma}{1+p_i/2} >0$, we choose $c_1 =c_1(\gamma,p_n)>0$ such that $1/\alpha_i -\gamma\geq c_1$ for all $i$.
Using the fact that
\begin{align*}
    \sum_{i\in I} \frac{1}{\alpha_i}|a_i|^{2\alpha_i} \geq (1-\frac{n}{C_0})|a|^2_\alpha \quad \text{ and }\quad \sum_{i\notin I} \frac{1}{\alpha_i}|a_i|^{2\alpha_i} \leq \frac{n}{C_0}|a|^2_\alpha,
\end{align*}
we have
\begin{align*}
    w^t\diag{|z_i|^{p_i/2}}Z\diag{|z_i|^{p_i/2}}w &\leq  \frac{L}{|a|_\alpha^{2-\gamma}} \frac{1}{|a|^4_\alpha} \left( - c_1 \left( \sum_{i \in I} \frac{1}{\alpha_i}|a_i|^{2\alpha_j}\right)^2 +  2\sum_{i} \frac{1}{\alpha_i}|a_i|^{2\alpha_j} \sum_{j \notin I} \frac{1}{\alpha_i}|a_i|^{2\alpha_j}\right) \\
    &\leq \frac{L}{|a|_\alpha^{2-\gamma}} \left( -c_1\left(1-\frac{n}{C_0} \right)^2 + \frac{2}{C_0}\right) \\
    &\leq -\frac{c_1}{2}\frac{L}{|a|_\alpha^{2-\gamma}}.
\end{align*}
by choosing large $C_0 =C_0(n,c_1)>n$.
Moreover, since
    $w^tw = \sum_{i\in I} \frac{1}{\alpha_i^2}\frac{|a|_\alpha^{2(k_i-1)}}{|a_i|^{2\alpha_i(k_i-1)}} \leq CC_0^{k_n-1}$, we obtain
\begin{align*}
    \frac{w^t\diag{|z_i|^{p_i/2}}Z\diag{|z_i|^{p_i/2}}w}{w^tw} \leq -\frac{c_1}{CC_0^{k_n-1}}\frac{L}{|a|_\alpha^{2-\gamma}}.
\end{align*}
For some large $C_2>1$, we define $\epsilon=(\epsilon_i)$ as
\begin{align} \label{epsilon}
    \frac{1}{\epsilon_i} = C_2 \left( 1+ \frac{L^{2\beta_i}|a_i|^{2(\alpha_i-1)}}{|a|_\alpha^{2-\gamma}}\right).
\end{align}
Then we have
\begin{align*}
     \diag{|z_i|^{p_i/2}}Z\diag{\epsilon_i} Z\diag{|z_i|^{p_i/2}} &=\frac{L}{C_2}\diag{\sgn{a_i} \frac{|a_i|^{\alpha_i k_i}}{|a|_\alpha^{k_i}}} A \diag{\frac{ L^{2\beta_i}|a_i|^{2(\alpha_i-1)}}{|a|^{2-\gamma}_\alpha+L^{2\beta_i}|a_i|^{2(\alpha_i-1)}}}A \diag{\sgn{a_i} \frac{|a_i|^{\alpha_i k_i}}{|a|_\alpha^{k_i}}}.
\end{align*}
Note that every matrix in right hand side is bounded, to find that
\begin{align*}
    \left|\frac{w^t\diag{|z_i|^{p_i/2}}2Z\diag{\epsilon_i} Z\diag{|z_i|^{p_i/2}}w}{w^tw} \right|\leq C\frac{L}{C_2} \leq \frac{c_1}{2CC_0^{k_n-1}}\frac{L}{|a|_\alpha^{2-\gamma}},
\end{align*}
by choosing $C_2=C_2(C_0,c_1,\gamma,p_n)>1$ sufficiently large.
Applying $\left(\diag{|z_i|^{p_i/2}}w,-\diag{|z_i|^{p_i/2}}w\right)$ to the matrix inequality \eqref{mtx}, we have
\begin{align*}
    \frac{w^t\diag{|z_i|^{p_i/2}}(X-Y)\diag{|z_i|^{p_i/2}}w}{w^tw} \leq \frac{w^t\diag{|z_i|^{p_i/2}}4Z^\epsilon\diag{|z_i|^{p_i/2}}w}{w^tw} \leq -\frac{1}{C}\frac{L}{|a|_\alpha^{2-\gamma}}.
\end{align*}
This implies that  $\diag{|z_i|^{p_i/2}}(X-Y)\diag{|z_i|^{p_i/2}}$ has at least one eigenvalue smaller than $-\frac{1}{C}\frac{L}{|a|_\alpha^{2-\gamma}}$.
Since $X-Y \leq 0$, there exists some $c_3>0$ such that
\begin{align*} 
    \mathcal{M}^+(\diag{|z_i|^{p_i/2}}(X-Y)\diag{|z_i|^{p_i/2}}) \leq -c_3\frac{L}{|a|_\alpha^{2-\gamma}}.
\end{align*}
Our goal is to prove that every other term in \eqref{allest} is `smaller' than $\frac{L}{|a|_\alpha^{2-\gamma}}$ for large $L>1$ to get a contradiction.
For example, using \eqref{zest} we get 
\begin{align} \label{ziest}
    |z_i| \leq CL^{\beta_i} \frac{|a|^{(2\alpha_i-1)/\alpha_i}_\alpha}{|a|^{2-\gamma}_\alpha}  \leq C \frac{L^{\beta_i}}{|a|_\alpha^{1/\alpha_i - \gamma}} ,
\end{align}
by which we have
\begin{align*}
    \sum_i 2\Lambda K|z_i|^{p_i} \leq C\sum_i\frac{L^{p_i\beta_i}}{|a|_\alpha^{p_i(1/\alpha_i-\gamma)}} = C\sum_i \frac{L^{\frac{p_i/2}{1+p_i/2}}}{|a|_\alpha^{\frac{p_i}{2+p_i}(2-2\gamma)}}.
\end{align*}
Since $\frac{p_i/2}{1+p_i/2} < 1,\  \frac{p_i}{2+p_i}(2-2\gamma) <2-\gamma$, each summand is smaller than $\frac{L}{C_L|a|_\alpha^{2-\gamma}}$ where $C_L>1$ is arbitrarily large provided $L>1$ is large enough.
Therefore, we obtain
\begin{align*}
    I_1 \leq -c_3\frac{L}{|a|_\alpha^{2-\gamma}} +\frac{L}{C_L|a|_\alpha^{2-\gamma}}.
\end{align*}
\textbf{Estimates of $I_{21},I_{22}$.}
We now estimate $|I_{21}|,|I_{22}|$ and show that they are smaller than $\frac{L}{|a|_\alpha^{2-\gamma}}$ for large $L>1$.
We write $d = K(x-x_0)$, then we have $|d| \leq K/2$. 
We also write $z^j := (z_1+d_1, \cdots, z_j+d_j,z_{j+1},\cdots,z_n)$ for $j=0,\cdots,n$.
Note that $z^0=z, \ z^n=z+d$.
By the assumption \ref{A2}, we have
\begin{align} \label{grad}
    |I_{21}| &= |F(X+KI,z+d,x)-F(X+KI,z,x)|  \nonumber\\
    &\leq \sum_i |F(X+KI,z^i,x)-F(X+KI,z^{i-1},x)|\nonumber \\
    &\leq \Lambda\sum_i\left||z_i+d_i|^{p_i}-|z_i|^{p_i}\right|(|X_{ii}|+K) \nonumber\\
    &+\Lambda\sum_{i< j} \left||z_i+d_i|^{p_i/2}-|z_i|^{p_i/2}\right||z_j|^{p_j/2}|X_{ij}| + \Lambda \sum_{i>j} \left||z_i+d_i|^{p_i/2}-|z_i|^{p_i/2}\right||z_j+d_j|^{p_j/2}|X_{ij}| .
\end{align}
We now estimates $|X_{ii}|,|X_{ij}|$ by the matrix inequality \eqref{mtx}.
Using \eqref{Zmtx} and \eqref{epsilon}, we have
\begin{align*}
    \diag{\sqrt{\epsilon_i}}Z\diag{\sqrt{\epsilon_i}} = \frac{1}{C_2} \diag{\frac{\sgn{a_i} L^{\beta_i} |a_i|^{\alpha_i-1}}{\sqrt{|a|^{2-\gamma}_\alpha+L^{2\beta_i}|a_i|^{2(\alpha_i-1)}}}} A\diag{\frac{\sgn{a_i}L^{\beta_i} |a_i|^{\alpha_i-1}}{\sqrt{|a|^{2-\gamma}_\alpha+L^{2\beta_i}|a_i|^{2(\alpha_i-1)}}}} .
\end{align*}
Since each matrix on the right-hand side is bounded, we have $\norm{\diag{\sqrt{\epsilon_i}}Z\diag{\sqrt{\epsilon_i}}} \leq C$.
Moreover, by \eqref{Zmtx} and the fact that
\begin{align*}
    Z\diag{\epsilon_i} Z &=\frac{1}{C_2}\diag{\sgn{a_i}L^{\beta_i} |a_i|^{\alpha_i-1}} A \diag{\frac{L^{2\beta_i}|a_i|^{2(\alpha_i-1)}}{|a|^{2-\gamma}_\alpha+L^{2\beta_i}|a_i|^{2(\alpha_i-1)}}}A \diag{\sgn{a_i}L^{\beta_i} |a_i|^{\alpha_i-1}},
\end{align*}
we have
\begin{align} \label{Zentry}
    |Z^\epsilon_{ij}|\leq \frac{C}{|a|_\alpha^{2-\gamma}}(L^{\beta_i} |a_i|^{\alpha_i-1})(L^{\beta_j} |a_j|^{\alpha_j-1}).
\end{align}
By applying $(e_i,0)$ to the matrix inequality \eqref{mtx}, we get $-C/\epsilon_{i} \leq X_{ii} \leq Z^\epsilon_{ii}$ which implies
\begin{align} \label{Xiiest}
    |X_{ii}| \leq C\left( 1+ \frac{L^{2\beta_i}|a_i|^{2(\alpha_i-1)}}{|a|_\alpha^{2-\gamma}}\right).
\end{align}
In order to get estimates of $|X_{ij}|$ for some fixed $i \neq j$, we apply $(b_je_i + b_ie_j,0)$ and $(b_je_i - b_ie_j,0)$ to the matrix inequality \eqref{mtx} with
\begin{align*}
    b_i = \sqrt{ 1+ \frac{L^{2\beta_i}|a_i|^{2(\alpha_i-1)}}{|a|_\alpha^{2-\gamma}} } , \qquad b_j = \sqrt{ 1+ \frac{L^{2\beta_j}|a_j|^{2(\alpha_j-1)}}{|a|_\alpha^{2-\gamma}} }.
\end{align*}
Then we obtain
\begin{align*}
    -Cb_j^2/\epsilon_{i} -Cb_i^2/\epsilon_{j} \leq b_j^2X_{ii} +2b_ib_jX_{ij} + b_i^2X_{jj} \leq b_j^2Z^\epsilon_{ii} +2b_ib_jZ^\epsilon_{ij} +b_i^2Z^\epsilon_{jj}, \\
    -Cb_j^2/\epsilon_{i} -Cb_i^2/\epsilon_{j} \leq b_j^2X_{ii} -2b_ib_jX_{ij} + b_i^2X_{jj} \leq b_j^2Z^\epsilon_{ii} -2b_ib_jZ^\epsilon_{ij} +b_i^2Z^\epsilon_{jj}.
\end{align*}
By subtracting the above inequalities, we get
\begin{align*}
    4|b_ib_jX_{ij}| \leq |b_j^2Z^\epsilon_{ii}| +2|b_ib_jZ^\epsilon_{ij}| +|b_i^2Z^\epsilon_{jj}| + Cb_j^2/\epsilon_{i} +Cb_i^2/\epsilon_{j}.
\end{align*}
Note that using \eqref{Zentry} and \eqref{epsilon}, we obtain
\begin{align*}
    |b_j^2Z^\epsilon_{ii}| +2|b_ib_jZ^\epsilon_{ij}| +|b_i^2Z^\epsilon_{jj}| + Cb_j^2/\epsilon_{i} +Cb_i^2/\epsilon_{j} \leq C\left(1+\frac{L^{2\beta_i}|a_i|^{2(\alpha_i-1)}}{|a|_\alpha^{2-\gamma}}\right) \left(1+\frac{L^{2\beta_j}|a_j|^{2(\alpha_j-1)}}{|a|_\alpha^{2-\gamma}}\right),
\end{align*}
which implies
\begin{align} \label{Xijest}
    |X_{ij}| \leq C\left( 1+ \frac{L^{\beta_i}|a_i|^{\alpha_i-1}}{|a|_\alpha^{1-\gamma/2}}\right) \left( 1+ \frac{L^{\beta_j}|a_j|^{\alpha_j-1}}{|a|_\alpha^{1-\gamma/2}}\right).
\end{align}
We denote $X_i := 1+ \frac{L^{\beta_i}|a_i|^{\alpha_i-1}}{|a|_\alpha^{1-\gamma/2}} \geq 1$ so that $|X_{ij}| \leq CX_iX_j$.
Note that \eqref{Xijest} for $i=j$ is also true by \eqref{Xiiest}.
For $0< \theta \leq p_i/2$, by using \eqref{ziest} and the fact $1/\alpha_i-\gamma>0$, we obtain
\begin{align} \label{ziXiest}
    |z_i|^{p_i/2-\theta}|X_i| &\leq C\frac{L^{\beta_i (p_i/2-\theta)} }{|a|_\alpha^{(1/\alpha_i-\gamma)(p_i/2-\theta)}}\left(1+ \frac{L^{\beta_i} |a|^{1-1/\alpha_i}_\alpha }{|a|_\alpha^{1-\gamma/2}} \right) \nonumber\\ 
    &\leq \frac{C|a|_\alpha^{(1/\alpha_i-\gamma)\theta}}{L^{\beta_i\theta}} \cdot \frac{L^{\beta_i(1+p_i/2)} }{|a|_\alpha^{(1/\alpha_i-\gamma)(1+p_i/2)+\gamma/2}} \nonumber\\
    &\leq \frac{L^{1/2} }{C_L|a|_\alpha^{1-\gamma/2}},
\end{align}
where $C_L>0$ becomes arbitrarily large when $L$ is large enough.
Observe that if $\theta=0$, then we just have
\begin{align} \label{ziXiest2}
    |z_i|^{p_i/2}|X_i| \leq C_4\frac{L^{1/2} }{|a|_\alpha^{1-\gamma/2}},
\end{align} where $C_4>0$ is independent of $L$.
We use the above estimates to get the estimate for $I_{21}$.
If $|z_i|\geq K$ and $p_i\geq1$, then by \eqref{ziXiest} we have
\begin{align*}
    J_1:=\left||z_i+d_i|^{p_i}-|z_i|^{p_i}\right|(|X_{ii}|+K) &\leq C |z_i|^{p_i-1}X_i^2 \\
    &\leq C(|z_i|^{p_i/2-1/2}X_i)^2 \\
    &\leq \frac{L}{C_L|a|_\alpha^{2-\gamma}}.
\end{align*}
If $|z_i| \leq K$ or $0<p_i \leq 1$, by $\left||z_i+d_i|^{p_i}-|z_i|^{p_i}\right| \leq C$ we get
\begin{align*}
    J_1 \leq C|X_{i}|^2 
    \leq \frac{L}{C_L|a|_\alpha^{2-\gamma}}.
\end{align*}
Note that for $p_i=0$, then we have just $J_1=0$.

In order to obtain the estimates for $I_{21}$ in \eqref{grad} with $|X_{ij}|$ parts, using $|X_{ij}| \le CX_iX_j$, we find the estimates of
$\left||z_i+d_i|^{p_i/2}-|z_i|^{p_i/2}\right|X_i$ and 
$|z_j|^{p_j/2}X_j$ (also $|z_j+d_j|^{p_j/2}X_j$) separately, and then product them.
If $|z_i| \geq K$ and $p_i\geq2$, then we have
\begin{align*}
    J_2:=\left||z_i+d_i|^{p_i/2}-|z_i|^{p_i/2}\right|X_i \leq C|z_i|^{p_i/2-1} X_i \leq \frac{L^{1/2}}{C_L|a|_\alpha^{1-\gamma/2}}.
\end{align*}
If $|z_i| \leq K$ or $0 < p_i \leq 2$, then we get
\begin{align*}
    J_2 \leq CX_{i} \leq \frac{L^{1/2}}{C_L|a|_\alpha^{1-\gamma/2}}.
\end{align*}
Also, we have $J_2=0$ when $p_i=0$.
Moreover, if $|z_j| \geq K$, then since $|z_j+d_j| \leq C|z_j|$, we use \eqref{ziXiest2} to have
\begin{align*}
    J_3:=|z_j+d_j|^{p_j/2}X_j \leq C|z_j|^{p_j/2} X_j \leq C_4\frac{L^{1/2}}{|a|_\alpha^{1-\gamma/2}}.
\end{align*}
Observe that $J_4=|z_j|^{p_j/2}X_j$ has the same estimates as above.
If $|z_j| \leq K$, then by $|z_j+d_j| \leq C$ we get
\begin{align*}
    J_3\leq C X_j \nonumber \leq C_4\frac{L^{1/2}}{|a|_\alpha^{1-\gamma/2}}.
\end{align*}
which also holds when $p_i=0$.
Note that $I_{21}$ can be bounded by \eqref{grad} which consists of $J_1$, $J_2 J_3$ and $J_2 J_4$.
Even though $J_3$ and $J_4$ may not be smaller than $\frac{L^{1/2}}{C|a|_\alpha^{1-\gamma/2}}$ for some large $C>1$, but because of $J_2$, the products of them are smaller than $\frac{L}{C_L|a|_\alpha^{2-\gamma}}$.
By the same computation for $I_{22}$, we obtain
\begin{align*}
    |I_{21}|, |I_{22}| \leq \frac{L}{C_L|a|_\alpha^{2-\gamma}}.
\end{align*}
\textbf{Estimates of $I_{3}$.} We now get the estimate for $I_{3}$.
Note that by \eqref{adelta}, we get 
$|\delta| = |x-y| \leq CL^{-\beta_n}$.
Using the assumption \ref{A3}, we have
\begin{align*}
    |I_3| &= |F(X+KI,z,x)-F(X+KI,z,y) | \\
    &\leq C|x-y|^\chi\sum_{i,j}|z_i|^{p_i/2}|z_j|^{p_j/2}|X_{ij}+K\delta_{ij}| \\
    &\leq CL^{-\beta_n\chi}\sum_{i,j}(|z_i|^{p_i/2}|X_i|)(|z_j|^{p_j/2}|X_{j}|) \\
    &\leq CL^{-\beta_n\chi}\frac{L }{|a|_\alpha^{2-\gamma}} \ \leq \frac{L }{C_L|a|_\alpha^{2-\gamma}},
\end{align*}
where we used $|X_{ij}| \leq CX_iX_j$ and \eqref{ziXiest2}.

\textbf{Conclusion.} Note that by \eqref{ziest} we have
\begin{align*}
    |I_4|&=C_f\left(2+\sum|(z_x)_i|^{p_i+1} + |(z_y)_i|^{p_i+1}\right) \leq C\left(1+\sum|z_i|^{p_i+1}\right) \\
    &\leq C\left(1+\sum_i\frac{L^{(p_i+1)\beta_i}}{|a|_\alpha^{(p_i+1)(1/\alpha_i-\gamma)}} \right)\leq  C \left( 1+ \sum_i \frac{L^{\frac{1/2+p_i/2}{1+p_i/2}}}{|a|_\alpha^{\frac{1/2+p_i/2}{1+p_i/2}(2-2\gamma)}} \right) \\
    &\leq \frac{L }{C_L|a|_\alpha^{2-\gamma}}.
\end{align*}
Gathering all the estimates above into \eqref{allest}, we obtain
\begin{align*}
       c_3\frac{L}{|a|_\alpha^{2-\gamma}} \leq -I_1 \leq |I_{21}| + |I_{22}| +|I_{3}| +|I_{4}| \leq \frac{L }{C_L|a|_\alpha^{2-\gamma}}.
\end{align*}
By choosing large enough $L>1$ so that $1/C_L<c_3  $, we get a contradiction.
Therefore, we conclude that $u$ is $C^\gamma$ for any $\gamma<1$.
\end{proof}

\section{Lipschitz estimates for anisotropic equations} \label{sec4}
We now prove the Lipschitz continuity of $u$ using the $ C^\gamma$-regularity for any $\gamma<1$ obtained in Section \ref{sec3}.
The proof is simpler than H\"older continuity; we just consider the isotropic distance function $|x|^2 = x_1^2 + \cdots + x_n^2$ and use the standard Ishii-Lions method.

\begin{proof}[Proof of Theorem \ref{Main}.] 
As before, we claim that 
\begin{align*}
    M= \max_{x,y\in \overline{B_1}} \left\{ u(x) - u(y) - L\phi(|x-y|) - \frac{K}{2}|x-x_0|^2 -\frac{K}{2}|y-x_0|^2 \right\} \leq 0,
\end{align*}
where
\begin{align*}
    \phi(t) = 
    \begin{cases}
        t-\frac{1}{1+\kappa}t^{1+\kappa} &\text{for } t\in[0,1] \\
        1-\frac{1}{1+\kappa} &\text{for } t>1,
    \end{cases}
\end{align*}
for some small $0<\kappa<1$.
Explicitly, we choose small $0<\kappa$ and $\gamma<1$ close to $1$ such that
\begin{align*}
    \min\{\gamma/4, \ p_m\gamma/4,  \ \chi\}  \geq (2+p_n)\kappa +p_n(1-\gamma), 
\end{align*}
where $p_m$ is the smallest number of $(p_i)$ which is not 0.
We prove this by contradiction assuming $M>0$ for any large $L, K>0$.
We have $x \neq y$ and since $u \in C^\gamma$ for any $\gamma<1$, we get
\begin{align*}
    L \phi(|x-y|) + \frac{K}{2}|x-x_0|^2 +\frac{K}{2}|y-x_0|^2  \leq \norm{u}_{C^\gamma}\delta^{\gamma},
\end{align*}
where $\delta=|x-y| \neq 0$.
By choosing large $K>32\norm{u}_{C^\gamma}$, we have $x,y \in B_1$ and
\begin{align} \label{xx0est}
    |x-x_0| < \frac{1}{4}\delta^{\gamma/2} \quad \text{ and } \quad  |y-y_0| < \frac{1}{4}\delta^{\gamma/2}.
\end{align}
Furthermore, since $L\phi(\delta) \leq L\delta$, we have
\begin{align*}
    \delta \leq CL^{-\frac{1}{1-\gamma}}.
\end{align*}
We apply the standard Ishii-Lions lemma to $u(x)- \frac{K}{2}|x-x_0|^2$ and $-u(y) -\frac{K}{2}|y-x_0|^2$, in order to find $X,Y \in S(n)$ such that
     \begin{align*}
         (z_x, X+KI)\in \overline{\mathcal{J}}^{2,+} \left( u(x)\right) \quad \text{and} \quad
         (z_y, Y-KI)\in \overline{\mathcal{J}}^{2,-} \left( u(y)\right),
     \end{align*}
     where
     \begin{align*}
         z_x = z+K(x-x_0) \quad \text{and} \quad z_y = z-K(y-x_0),
     \end{align*}
    with
     \begin{align} \label{zest2}
         z= L\phi'(|x-y|)\frac{x-y}{|x-y|} = L(1-\delta^{\kappa}) a \quad \text{where} \quad a= \frac{x-y}{|x-y|}.
         \end{align}   
Moreover, for any small $\epsilon< 1$, we can choose $X,Y \in S(n)$ satisfying
\begin{align} \label{mtx2}
    -\left(\frac{1}{\epsilon} + \norm{Z}\right) 
    \begin{pmatrix}
        I & 0 \\
        0 &I
    \end{pmatrix}
    \leq
    \begin{pmatrix}
        X & 0 \\
        0 &-Y
    \end{pmatrix}
    \leq
    \begin{pmatrix}
        Z^\epsilon & -Z^\epsilon \\
        -Z^\epsilon & Z^\epsilon
    \end{pmatrix},
\end{align}
where
\begin{align*}
    Z =\frac{L}{\delta}\left( -\kappa\delta^{\kappa} a\otimes a + (1-\delta^{\kappa})(I-a\otimes a) \right),
\end{align*}
and $Z^\epsilon = Z+ 2\epsilon Z^2$.
Note that $\norm{Z} \leq 2L/\delta$.
By the definition of viscosity solution, we have $F(X+KI, z_x, x) \geq f(x,z_x)$ and $F(Y-KI, z_y, y) \leq f(y,z_y)$.
We use the same computation as in \eqref{allest} and estimates $I_1$, $I_{21}$, $I_{22}$, $I_{3}$ and $I_{4}$.

\textbf{Estimates of $I_1$.}
We first find the estimate for $I_1$. By the assumption \ref{A1},
\begin{align*}
    I_1 &= F(X+KI,z,y)-F(Y-KI,z,y) \\
    &\leq \mathcal{M}^+(\diag{|z_i|^{p_i/2}}(X-Y)\diag{|z_i|^{p_i/2}}) + \sum_i 2\Lambda K|z_i|^{p_i}.
\end{align*}
Likewise before, we have $X-Y \leq 0$, and so $\diag{|z_i|^{p_i/2}}(X-Y)\diag{|z_i|^{p_i/2}} \leq 0$ and all the eigenvalues are nonpositive.
For some $C_0>2n$, we define
\begin{align*}
    I = I(a,C_0) = \left\{ i \in \{ 1,\cdots,n\} : |a_i|^{2} \geq \frac{\delta^{\kappa}}{C_0}\right\}.
\end{align*}
Since $|a| = 1$, $I$ is nonempty. We define 
\begin{align*}
    w= \sum_{i \in I} \frac{a_i}{(L(1-\delta^{\kappa})|a_i|)^{p_i/2}} e_i.
\end{align*}

Then using \eqref{zest2} and the fact that $\sum_i |a_i|^2 =1$, we get
\begin{align*}
    w^t\diag{|z_i|^{p_i/2}}Z\diag{|z_i|^{p_i/2}}w &= \frac{L}{\delta}\left( \sum_{i \in I} a_ie_i \right)\left( -\kappa\delta^{\kappa} a\otimes a + (1-\delta^{\kappa})(I-a\otimes a) \right)\left( \sum_{i \in I} a_ie_i \right) \\
    &= \frac{L}{\delta}\left( \sum_{i \in I} |a_i|^2 \right)\left( -\kappa\delta^{\kappa}  \sum_{i \in I} |a_i|^2  + (1-\delta^{\kappa})\left(1-\sum_{i \in I} |a_i|^2 \right) \right) \\
    &= \frac{L}{\delta}\left( \sum_{i \in I} |a_i|^2 \right)\left( -\kappa\delta^{\kappa} \sum_{i \in I} |a_i|^2 + (1-\delta^{\kappa})\sum_{i \notin I} |a_i|^2 \right) \\
    &\leq \frac{L}{\delta}\left( \sum_{i \in I} |a_i|^2 \right)\left( -\kappa\delta^{\kappa} \sum_{i \in I} |a_i|^2 + \frac{n\delta^{\kappa}}{C_0} \right) \\
    &\leq \frac{L}{\delta^{1-\kappa}}\left( 1-\frac{ n\delta^{\kappa}}{C_0}\right) \left( -\kappa \left( 1-\frac{ n\delta^{\kappa}}{C_0}\right) + \frac{n}{C_0} \right) \\
    &\leq -\frac{\kappa}{2}\frac{L}{\delta^{1-\kappa}},
\end{align*}
by choosing large $C_0=C_0(n,\kappa)>2n$.
Since
\begin{align*}
    w^tw = \sum_{i\in I} \frac{|a_i|^2}{(L(1-\delta^{\kappa})|a_i|)^{p_i}} \leq \frac{CC_0^{p_n/2}}{\delta^{\kappa p_n/2}},
\end{align*}
we get
\begin{align*}
    \frac{w^t\diag{|z_i|^{p_i/2}}Z\diag{|z_i|^{p_i/2}}w}{w^tw} \leq -c_1\frac{L}{\delta^{1-(1+p_n/2)\kappa}},
\end{align*}
for some small $c_1 = c_1(n,\kappa,p_n,C_0)>0$.
For some large $C_2>1$, we define $\epsilon= \frac{\delta^{1+(1+p_n/2)\kappa}}{C_2L}>0$.
Then we get
\begin{align*}
    \epsilon Z^2 = \frac{L}{C_2\delta^{1-(1+p_n/2)\kappa}}\left( -\kappa\delta^{\kappa} a\otimes a + (1-\delta^{\kappa})(I-a\otimes a) \right)^2.
\end{align*}
Since the above matrix is bounded, we have 
\begin{align*}
    \left|\frac{w^t\diag{|z_i|^{p_i/2}}\epsilon Z^2\diag{|z_i|^{p_i/2}}w}{w^tw} \right|\leq \frac{L}{C_2\delta^{1-(1+p_n/2)\kappa}}.
\end{align*}
Therefore, applying $\left(\diag{|z_i|^{p_i/2}}w,-\diag{|z_i|^{p_i/2}}w\right)$ to the matrix inequality \eqref{mtx2}, we obtain
\begin{align*}
    \frac{w^t\diag{|z_i|^{p_i/2}}(X-Y)\diag{|z_i|^{p_i/2}}w}{w^tw} \leq \frac{w^t\diag{|z_i|^{p_i/2}}4Z^\epsilon\diag{|z_i|^{p_i/2}}w}{w^tw} \leq -\frac{L}{C\delta^{1-(1+p_n/2)\kappa}},
\end{align*}
where $C_2=C_2(c_1)>1$ is large enough.
Thus, there exists $c_3>0$ such that
\begin{align*} 
    \mathcal{M}^+(\diag{|z_i|^{p_i/2}}(X-Y)\diag{|z_i|^{p_i/2}}) \leq - c_3\frac{L}{\delta^{1-(1+p_n/2)\kappa}}.
\end{align*}
Note that by \eqref{zest2} we have $|z_i| \leq L$ so that 
\begin{align*}
    \sum_i 2\Lambda K|z_i|^{p_i} &\leq CL^{p_n} \leq CL^{p_n-1}\delta^{1-(1+p_n/2)\kappa} \frac{L}{\delta^{1-(1+p_n/2)\kappa}} \\
    &\leq CL^{-\frac{1}{1-\gamma}(1-(1+p_n/2)\kappa - (p_n-1) (1-\gamma))}\frac{L}{\delta^{1-(1+p_n/2)\kappa}} \\
    &\leq \frac{L}{C_L\delta^{1-(1+p_n/2)\kappa}},
\end{align*}
where we use $1-(1+p_n/2)\kappa - (p_n-1) (1-\gamma)>0$.
Therefore, we obtain
\begin{align*}
    I_1 \leq -c_3\frac{L}{\delta^{1-(1+p_n/2)\kappa}} + \frac{L}{C_L\delta^{1-(1+p_n/2)\kappa}}.
\end{align*}
\textbf{Estimates of $I_{21},I_{22}$.}
We now estimate $|I_{21}|,|I_{22}|$ and show that they are `smaller' than $\frac{L}{C_L\delta^{1-(1+p_n/2)\kappa}}$.
We write $d = K(x-x_0)$, then from \eqref{xx0est}, we have $|d_i| \leq \delta^{\gamma/2} K/2$. 
By the assumption \ref{A2}, we have
\begin{align} \label{grad2}
    |I_{21}| &= |F(X+KI,z+d,x)-F(X+KI,z,x)|  \nonumber\\
    &\leq \Lambda\sum_i\left||z_i+d_i|^{p_i}-|z_i|^{p_i}\right|(|X_{ii}|+K) \nonumber\\
    &+\Lambda\sum_{i< j} \left||z_i+d_i|^{p_i/2}-|z_i|^{p_i/2}\right||z_j|^{p_j/2}|X_{ij}| + \Lambda \sum_{i>j} \left||z_i+d_i|^{p_i/2}-|z_i|^{p_i/2}\right||z_j+d_j|^{p_j/2}|X_{ij}|. 
\end{align}
Note that $\frac{1}{\epsilon} + \norm{Z} \leq \frac{CL}{\delta^{1+(1+p_n/2)\kappa}}$ and $|Z_{ij}^\epsilon| \leq \frac{CL}{\delta}$, to find
\begin{align} \label{Xijest2}
    |X_{ij}| \leq \frac{CL}{\delta^{1+(1+p_n/2)\kappa}}.
\end{align}
For some $\eta \in  \mathbb{R}$ and $ \theta > (2+p_n)\kappa + \eta(1-\gamma)$, we obtain
\begin{align} \label{dLXijest}
    \delta^\theta L^\eta(|X_{ij}|+1) &\leq  CL^\eta \delta^{\theta - (2+p_n)\kappa}\frac{L}{\delta^{1-(1+p_n/2)\kappa}} \nonumber \\
    &\leq CL^{-\frac{1}{1-\gamma}(\theta -(2+p_n)\kappa - \eta (1-\gamma))} \frac{L}{\delta^{1-(1+p_n/2)\kappa}} \nonumber \\
    &\leq \frac{L}{C_L\delta^{1-(1+p_n/2)\kappa}},
\end{align}
where $C_L>0$ becomes arbitrarily large when $L$ is large enough.
Furthermore, since $|d_i| \leq \delta^{\gamma/2} K/2$, we have
\begin{align*}
    \left||z_i+d_i|^{p_i}-|z_i|^{p_i}\right| \leq 
    \begin{cases}
        C\delta^{\gamma/2} |z_i|^{p_i-1} \qquad &\text{ if } |z_i| \geq \delta^{\gamma/2} K \text{ and } p_i \geq 1, \\
        C\delta^{p_i\gamma/2} \qquad &\text{ if } |z_i| \leq \delta^{\gamma /2} K \text{ or } 0<p_i \leq 1.
    \end{cases}
\end{align*}
If $p_i=0$, then the above quantity is just 0.
If $|z_i|\geq \delta^{\gamma/2} K$ and $p_i\geq1$, then since $|z_i| \leq L$, we have
\begin{align*}
    J_1:=\left||z_i+d_i|^{p_i}-|z_i|^{p_i}\right|(|X_{ii}|+K) \leq C\delta^{\gamma/2} L^{p_i-1}(|X_{ii}|+1) 
    \leq \frac{L}{C_L\delta^{1-(1+p_n/2)\kappa}}, 
\end{align*}
where we used \eqref{dLXijest} and the fact $\gamma/2  > (2+p_n)\kappa +p_n(1-\gamma)$.
If $|z_i|\leq \delta^{\gamma/2} K$ or $0<p_i\leq 1$, then
\begin{align*}
    J_1 \leq C\delta^{p_i\gamma /2}(|X_{ii}|+1)
    \leq \frac{L}{C_L\delta^{1-(1+p_n/2)\kappa}} ,
\end{align*}
since $p_m\gamma/2 > (2+p_n)\kappa $.
Recall that $p_m$ is the smallest number of $(p_i)$ which is not $0$.
If $|z_i|\geq \delta^{\gamma/2} K$, $p_i\geq2$ and $|z_j|\geq \delta^{\gamma/2} K$, then since $ |z_j+d_j| \leq C|z_j|$, we have
\begin{align*}
    J_2:=\left||z_i+d_i|^{p_i/2}-|z_i|^{p_i/2}\right||z_j+d_j|^{p_j/2}|X_{ij}| \leq C\delta^{\gamma/2}L^{p_i/2+p_j/2-1}|X_{ij}| 
    \leq \frac{L}{C_L\delta^{1-(1+p_n/2)\kappa}} .
\end{align*}
If $|z_i| \leq \delta^{\gamma/2} K$ or $0<p_i \leq 2$ and if $|z_j| \geq \delta^{\gamma/2} K$, then
\begin{align*}
    J_2 \leq C\delta^{p_i\gamma/4}L^{p_j/2}|X_{ij}| 
    \leq \frac{L}{C_L\delta^{1-(1+p_n/2)\kappa}} ,
\end{align*}
since $p_m\gamma/4  >(2+p_n)\kappa +p_n(1-\gamma)/2 $.
Note that the above two estimates are also true for $J_3:= \left||z_i+d_i|^{p_i/2}-|z_i|^{p_i/2}\right||z_j|^{p_j/2}|X_{ij}|$.
If $|z_i| \geq \delta^{\gamma/2} K$, $p_i\geq2$ and $|z_j| \leq \delta^{\gamma/2} K$, then
\begin{align*}
    J_2 \leq C\delta^{\gamma/2+p_j\gamma/4}L^{p_i/2-1}|X_{ij}| 
    \leq \frac{L}{C_L\delta^{1-(1+p_n/2)\kappa}}.
\end{align*}
If $|z_i| \leq \delta^{\gamma/2} K$ or $0<p_i\leq2$ and if $|z_j| \leq \delta^{\gamma/2} K$, then
\begin{align*}
    J_2\leq C\delta^{(p_i+p_j)\gamma/4}|X_{ij}| 
    \leq \frac{L}{C_L\delta^{1-(1+p_n/2)\kappa}}.
\end{align*}
Recall that $I_{21}$ can be bounded by \eqref{grad} which consists of $J_1$, $J_2$ and $J_3$, and $I_{22}$ can be estimated by similar computation.
Therefore, we obtain
\begin{align*}
    |I_{21}|, |I_{22}| \leq \frac{L}{C_L\delta^{1-(1+p_n/2)\kappa}}. 
\end{align*}
\textbf{Estimates of $I_{3}$.} We now get the estimate for $I_{3}$.
Using the assumption \ref{A3}, we have
\begin{align*}
    |I_3| &= |F(X+KI,z,x)-F(X+KI,z,y) | \\
    &\leq C|x-y|^\chi\sum_{i,j}|z_i|^{p_i/2}|z_j|^{p_j/2}|X_{ij}+K\delta_{ij}| \\
    &\leq C\delta^{\chi}\sum_{i,j}L^{p_n}(|X_{ij}|+1) \\
    &\leq \frac{L}{C_L\delta^{1-(1+p_n/2)\kappa}},
\end{align*}
by the fact $\chi > (2+p_n)\kappa + p_n(1-\gamma) $ and \eqref{dLXijest}.

\textbf{Conclusion.} Note that by \eqref{ziest} we have
\begin{align*}
    |I_4|&=C_f\left(2+\sum|(z_x)_i|^{p_i+1} + |(z_y)_i|^{p_i+1}\right) \leq C\left(1+\sum L^{p_i+1}\right) \\
    &\leq CL^{-\frac{1}{1-\gamma}(1-(1+p_n/2)\kappa - p_n (1-\gamma))}\frac{L}{\delta^{1-(1+p_n/2)\kappa}} \\
    &\leq \frac{L}{C_L\delta^{1-(1+p_n/2)\kappa}},
\end{align*}
since $1-(1+p_n/2)\kappa - p_n (1-\gamma)>0$.
Combining all the estimates above with \eqref{allest}, we obtain
\begin{align*}
       c_3\frac{L}{\delta^{1-(1+p_n/2)\kappa}} \leq -I_1 \leq |I_{21}| + |I_{22}| +|I_{3}| +|I_{4}| \leq \frac{L}{C_L\delta^{1-(1+p_n/2)\kappa}}.
\end{align*}
By choosing large enough $L>1$ so that $1/C_L<c_3 $, we finally reach a contradiction.
Consequently, we conclude that $u$ is Lipschitz continuous.

\end{proof}  

\section{The anisotropic Ishii-Lion Lemma} \label{sec5}
In this section, we prove the anisotropic Ishii-Lion Lemma (Lemma \ref{IL}).
In fact, Lemma \ref{IL} can be proved by a simple scaling of the original Ishii-Lion Lemma (Lemma \ref{ILo}), which is written below.
The proof of the Lemma \ref{ILo} can be found in \cite{Ishii92,Katzourakis15}.
\begin{lem} \label{ILo} (The original Ishii-Lion Lemma)
    Let $u,v \in C(\Omega)$ and $\phi(x,y) \in C^2(\Omega\times\Omega)$.
    Assume that $(\overline{x},\overline{y}) \in \Omega\times\Omega$ is a local maximum point of $u(x)+v(y) -\phi(x,y)$.
    Then for any $\epsilon>0$, there exist $X,Y \in S(n)$ such that
    \begin{align*}
        (D_x\phi(\overline{x},\overline{y}),X) \in \overline{\mathcal{J}^{2,+}_\Omega} u(\overline{x}), \quad (D_y\phi(\overline{x},\overline{y}),Y) \in \overline{\mathcal{J}^{2,-}_\Omega} v(\overline{y})
    \end{align*}
    and
    \begin{align*}
        -\left(\frac{1}{\epsilon}+\norm{Z} \right) I
         \leq 
        \begin{pmatrix}
             X & 0 \\
             0 &Y
         \end{pmatrix}
         \leq Z +\epsilon Z^2,
    \end{align*}
    where $Z = D^2\phi(\overline{x},\overline{y})$.
\end{lem} 
\begin{proof}[Proof of Lemma \ref{IL}] 
For any $\epsilon = (\epsilon_i)$ with $\epsilon_i>0$, we define $\tilde{u}, \tilde{v} \in C$ and $\tilde{\phi} \in C^2$ as
\begin{align*}
    \tilde{u}(x) = u(\diag{\sqrt{\epsilon_i}}x), \quad \tilde{v}(x) = v(\diag{\sqrt{\epsilon_i}}x), \quad \tilde{\phi}(x,y) = \phi(\diag{\sqrt{\epsilon_i}}x,\diag{\sqrt{\epsilon_i}}y).
\end{align*}
If $(\overline{x},\overline{y}) \in \Omega\times\Omega$ is a local maximum point of $u(x)+v(y) -\phi(x,y)$, then $(\tilde{x},\tilde{y}) = (\diag{1/\sqrt{\epsilon_i}}\overline{x}, \diag{1/\sqrt{\epsilon_i}}\overline{y})$ is a local maximum point of $\tilde{u}(x)+\tilde{v}(y) -\tilde{\phi}(x,y)$.
Therefore, we can apply the original Ishii-Lions (Lemma \ref{ILo} with $\epsilon=1$) to $\tilde{u}, \tilde{v}$ and $\tilde{\phi}$.
We know that there exist $\tilde{X},\tilde{Y} \in S(n)$ such that
\begin{align*}
    (D_x\tilde{\phi}(\tilde{x},\tilde{y}),\tilde{X}) \in \overline{\mathcal{J}^{2,+}} \tilde{u}(\tilde{x}), \quad (D_y\tilde{\phi}(\tilde{x},\tilde{y}),\tilde{Y}) \in \overline{\mathcal{J}^{2,-}} \tilde{v}(\tilde{y})
\end{align*}
and
\begin{align*}
    -\left(1+\norm{\tilde{Z}} \right) I
    \leq 
    \begin{pmatrix}
        \tilde{X} & 0 \\
        0 &\tilde{Y}
    \end{pmatrix}
    \leq \tilde{Z} +\tilde{Z}^2,
\end{align*}
where $\tilde{Z} = D^2\tilde{\phi}(\tilde{x},\tilde{y})$.
Note that
\begin{align*}
    D_x\tilde{\phi}(\tilde{x},\tilde{y}) = \diag{\sqrt{\epsilon_i}}D_x\phi(\overline{x},\overline{y}), \quad D_y\tilde{\phi}(\tilde{x},\tilde{y}) = \diag{\sqrt{\epsilon_i}}D_y\phi(\overline{x},\overline{y}) 
\end{align*}
and
\begin{align*}
    \tilde{Z} = \diag{\sqrt{\epsilon_i},\sqrt{\epsilon_i}}Z\diag{\sqrt{\epsilon_i},\sqrt{\epsilon_i}} \quad \text{ with } \quad Z= D^2\phi(\overline{x},\overline{y}).
\end{align*}
Therefore, by scaling back, we have
\begin{align*}
    (D_x\phi(\overline{x},\overline{y}),X) \in \overline{\mathcal{J}^{2,+}_\Omega} u(\overline{x}), \quad (D_y\phi(\overline{x},\overline{y}),Y) \in \overline{\mathcal{J}^{2,-}_\Omega} v(\overline{y})
\end{align*}
where $X = \diag{1/\sqrt{\epsilon_i}}\tilde{X}\diag{1/\sqrt{\epsilon_i}}$ and $Y = \diag{1/\sqrt{\epsilon_i}}\tilde{Y}\diag{1/\sqrt{\epsilon_i}}$.
Note that $$\tilde{Z} +\tilde{Z}^2 = \diag{\sqrt{\epsilon_i},\sqrt{\epsilon_i}}(Z+Z\diag{\epsilon_i,\epsilon_i}Z)\diag{\sqrt{\epsilon_i},\sqrt{\epsilon_i}},$$
to discover that
\begin{align*}
    -\left(1+\norm{\diag{\sqrt{\epsilon_i},\sqrt{\epsilon_i}}Z\diag{\sqrt{\epsilon_i},\sqrt{\epsilon_i}}} \right) I
    \leq 
    \begin{pmatrix}
        \diag{\sqrt{\epsilon_i}}X\diag{\sqrt{\epsilon_i}} & 0 \\
        0 &\diag{\sqrt{\epsilon_i}}Y\diag{\sqrt{\epsilon_i}}
    \end{pmatrix}
    \leq \diag{\sqrt{\epsilon_i},\sqrt{\epsilon_i}}(Z+Z\diag{\epsilon_i,\epsilon_i}Z)\diag{\sqrt{\epsilon_i},\sqrt{\epsilon_i}}.
\end{align*}
Finally, applying $\diag{1/\sqrt{\epsilon_i},1/\sqrt{\epsilon_i}}$ to the above matrix inequality, we get
\begin{align*}
    -\left(1+\norm{\diag{\sqrt{\epsilon_i},\sqrt{\epsilon_i}}Z\diag{\sqrt{\epsilon_i},\sqrt{\epsilon_i}}} \right) \diag{1/\epsilon_i,1/\epsilon_i}
    \leq 
    \begin{pmatrix}
        X & 0 \\
        0 & Y
    \end{pmatrix}
    \leq Z+Z\diag{\epsilon_i,\epsilon_i}Z,
\end{align*}
which completes the proof.
\end{proof}

\bibliographystyle{amsplain}

\end{document}